\documentclass[11pt, twoside]{article}
\usepackage{amssymb, amsmath, amsthm}
\usepackage[left=22mm, right=22mm, bottom=30mm]{geometry}
\usepackage{inputenc}
\usepackage{fancyhdr}
\usepackage{tikz}
\usetikzlibrary{positioning}
\usepackage{titling}
\usepackage{url}
\usepackage{hyperref}
\predate{}
\postdate{}
\usepackage{lipsum}
\newtheorem{theorem}{Theorem}
\newtheorem{lemma}[theorem]{Lemma}
\newtheorem{proposition}[theorem]{Proposition}

\newcommand{\Frac}{\hbox{\rm frac}}
\begin{document}
\newcommand{\Addresses}{{
\bigskip
\footnotesize
\medskip

Neil~Hindman, \textsc{Department of Mathematics, Howard University, Washington D.C., 20059, USA.}\par\nopagebreak\textit{Email address: }\texttt{nhindman@aol.com}

\medskip

Maria-Romina~Ivan, \textsc{Department of Pure Mathematics and Mathematical Statistics, Centre for Mathematical Sciences, Wilberforce Road, Cambridge, CB3 0WB, UK.}\par\nopagebreak\textit{Email address: }\texttt{mri25@dpmms.cam.ac.uk}

\medskip

Imre~Leader, \textsc{Department of Pure Mathematics and Mathematical Statistics, Centre for Mathematical Sciences, Wilberforce Road, Cambridge, CB3 0WB, UK.}\par\nopagebreak\textit{Email address: }\texttt{i.leader@dpmms.cam.ac.uk}
}
\medskip}
\pagestyle{fancy}
\fancyhf{}
\fancyhead [LE, RO] {\thepage}
\fancyhead [CE] {NEIL HINDMAN, MARIA-ROMINA IVAN AND IMRE LEADER}
\fancyhead [CO] {\small SOME NEW RESULTS ON MONOCHROMATIC SUMS AND PRODUCTS IN THE RATIONALS}
\renewcommand{\headrulewidth}{0pt}
\renewcommand{\l}{\rule{6em}{1pt}\ }
\title{\Large\textbf{SOME NEW RESULTS ON MONOCHROMATIC SUMS AND PRODUCTS IN THE RATIONALS}}
\author{NEIL HINDMAN, MARIA-ROMINA IVAN AND IMRE LEADER}
\date{}
\maketitle
\begin{abstract}
Our aim in this paper is to show that, for any $k$, there is a finite colouring of the set of rationals whose denominators
contain only the first $k$ primes such that no infinite set has all of its finite sums and products monochromatic.
We actually prove a `uniform' form of this: there is a finite colouring of the rationals with the property that no infinite set whose
denominators contain only finitely many primes has all of its finite sums and products monochromatic. We also give various other
results, including a new short proof of the old result that there is a finite colouring of the naturals such that no infinite set has all of its  
pairwise sums and products monochromatic.
\end{abstract}
\section{Introduction}
The Finite Sums Theorem \cite{FST} states that whenever the natural
numbers are finitely coloured there exists an infinite sequence
all of whose finite sums are the same colour. By considering just
powers of 2, this immediately implies the corresponding result for
products: whenever the naturals are finitely coloured there
exists a sequence all of whose products are the same
colour. But what happens if we want to combine sums and
products?
\par Hindman \cite{H} showed that one cannot ask for sums and products,
even just pairwise: there is a finite colouring of the naturals
for which no (injective) sequence has the set of all of its
pairwise sums and products monochromatic. The question of what
happens if we move from the naturals to a larger space 
is of especial interest. Bergelson, Hindman and Leader \cite{BHL}
showed that if we have a finite colouring of the reals with each
colour class measurable then there exist a sequence with the set
of all of its finite sums and products monochromatic. (They actually proved a
stronger statement: one may insist that the infinite sums are
the same colour as well.) However,
they also showed that there is a finite colouring of the dyadic
rationals such that no sequence has all of its finite
sums and products monochromatic. The questions of what happens in
general for finite colourings, in the rationals or the reals,
remain open.
\par The arguments in \cite{BHL} do not extend beyond the dyadics. Our aim
in this paper is to go further. Let $\mathbb Q_{(k)}$ denote the set of 
rationals whose denominators (in reduced form) involve only the first $k$ primes. Then we show that there is a finite colouring of $\mathbb Q_{(k)}$ such that no sequence has all of its
finite sums and 
products monochromatic.
\par In fact, we strengthen this result in two ways. First of all,
we insist that the number of colours does not grow with $k$, and
more importantly we give one colouring that `works for
all $\mathbb Q_{(k)}$ at once'. The actual statement is: there is a finite
colouring of the rationals such that no sequence for which the
set of primes that appear in the denominators is finite has the
set of its finite sums and products monochromatic. This is
really made up of two separate results: one about just pairwise 
sums, asserting that no such {\em bounded} sequence can have all of
its pairwise sums and products monochromatic, and the other about
general finite sums, saying that no such {\em unbounded} sequence can
have all of its finite sums and products monochromatic.

\par Our proofs are based on a careful analysis of the structure
of addition and multiplication in $\mathbb Q_{(k)}$, and also on a result
(Lemma~\ref{Marialemma1} below) about colouring pairs of naturals
that may be of independent interest. One 
application of this lemma is a new short proof of the
result of Hindman mentioned above, about pairwise sums
and products in the naturals.
\par We also prove various other related results. For example, we 
give a finite colouring of the reals such that no sequence that
is bounded and bounded away from zero can have its pairwise
sums and products monochromatic. 
\par The plan of the paper is as follows. In Section 2 we state and
prove our lemma about colouring pairs of naturals, and use it in Section 3 to give a new proof of the result about pairwise sums and products in the naturals. In Section 4 we give the above result about the reals, which we then build on in Section 5 to prove the statement about pairwise sums and products in 
bounded sequences. Amusingly, it is not entirely clear that
the colouring in Section 5 does not prevent monochromatic 
finite sums and products from {\em every} sequence in the 
rationals, and so we digress in Section 6 to exhibit such a
sequence for this colouring.  Finally in Section 7 we construct a colouring of the rationals such that if a sequence has the set of its finite sums and products monochromatic and the set of primes that appear in the denominators of its terms is finite, then the sequence has to be bounded -- together with the results of
Section 5 this establishes the main result.
\par Our notation is standard. We restrict our attention to the
positive rationals and the positive reals (which we write as $\mathbb Q^+$ and $\mathbb R^+$ respectively), since in all situations  
either it would
be impossible to use negative values (for example because the
sums are negative but the products are positive) or because, if
say we are dealing only with sums, then any colouring of the positive values could be reflected, using new colours, to the
negative values. Throughout the paper $\mathbb N$ is the set of positive integers.
\par We end this introduction by mentioning that in the case of
finite sequences very little is known. The question of
whether or not in every finite colouring of the naturals there
exist two (distinct) numbers that, together with their sum
and product, all have the same colour, remains tantalisingly
open. Moreira \cite{M} showed that we may find $x$ and $y$ such
that all of $x,x+y,xy$ have the same colour, and in the
rationals Bowen and Sabok \cite{MB} showed that we can indeed find the full
set $x,y,x+y,xy$. But for example for sums and products from
a set of size three or more almost nothing is known.
\section{Some useful lemmas}
In this section we prove the lemma mentioned above that we will make use of several
times (Lemma~\ref{Marialemma1}). We will also need two slight variants of it, namely Lemma~\ref{Marialemma2} and Lemma~\ref{Marialemma3}.
\begin{lemma}\label{Marialemma1}There exists a finite colouring $\Phi$ of $\mathbb N^{(2)}=\{(a,b)\in\mathbb N\times\mathbb N:a<b\}$ such that we cannot find two strictly increasing sequences of naturals, $(a_n)_{n\geq1}$ and $(b_n)_{n\geq1}$, such that $a_i<b_i$ for every $i$ and $\{(a_n+a_m,b_n+b_m):n<m\}\cup\{(a_n,b_m):n<m\}$ is monochromatic.
\end{lemma}

The way this will be of use to us is, roughly speaking, as follows. Suppose that we are trying to show that a certain kind of sequence cannot have its pairwise sums and products monochromatic (in the sense that there is a colouring that prevents this). Then it is enough to find two `parameters' $a$ and $b$ so that when we multiply two terms of the sequence the $a$-values and the
$b$-values add, but when we add two terms the resulting $a$-value is the $a$-value of the earlier term and the $b$-value
is the $b$-value of the later term.

Before starting the proof, we need a little notation. When a natural number is
written in binary we call the rightmost 1 the `last digit' of the number (the end), and the leftmost 1 the `first digit' of the number (the start). So for example the
number $10001010$ has start 7 and end 1. Also, we say that natural numbers $a$ and
$b$ are `right to left disjoint' if the end of $b$ is greater than the start of $a$.

\begin{proof} We colour a pair $(a,b)$ by $(c_1, c_2, c_3, c_4, c_5)$, where $c_1$ is the position of the last digit of $a \text{ mod 2 }$, $c_2$ is the position of the last digit of $b \text{ mod 2 }$, and $c_3$ and $c_4$ are the digits immediately to the left of the last digits of $a$ and $b$ respectively. Finally $c_5$ is 0 if the supports of $a$ and $b$ are right to left disjoint, and 1 otherwise.\par Suppose for a contradiction that we can find two sequences $(a_n)_{n\geq 1}$ and $(b_n)_{n\geq 1}$ as given in the statement of the lemma. Assume that for some $n<m$, $a_n$ and $a_m$ end at the same position. Say that position is $i$. Because $(a_n, b_m)$ and $(a_m, b_{m+1})$ have to have the same colour, it follows that $a_n$ and $a_m$ have the same last 2 digits. This implies that the position of the last digit of $a_n+a_m$ is $i+1$. On the other hand $(a_n, b_m)$ and $(a_n+a_m, b_n+b_m)$ must have the same colour, but they have a different $c_1$, a contradiction. Therefore we know that all $a_n$ have to end at different positions. By passing to subsequences, we may assume that the $a_n$ have pairwise right to left disjoint supports.\par Since $(a_n, b_m)$ and $(a_{n-1}, b_n)$ have the same colour, the same argument as above shows that for any $1<n<m$, $b_n$ and $b_m$ must end at different positions. Thus by passing to subsequences we may assume that both $a_n$ have right to left disjoint supports and $b_n$ have right to left disjoint supports.\\Finally, we can choose $n$ large enough that $a_1$ and $b_n$ have right to left disjoint supports and $b_1$ and $a_n$ have right to left disjoint supports. Thus $c_5=0$ for the pair $(a_1, b_n)$, but $c_5=1$ for the pair $(a_1+a_n,b_1+b_n)$ (as the right-hand side starts before the left-hand side finishes), a contradiction.
\end{proof}

We will also need two slight variants of this lemma.

\begin{lemma}\label{Marialemma2}There exists a finite colouring $\Psi$ of $\mathbb N^{(2)}$ such that we cannot find two strictly increasing sequences of naturals, $(a_n)_{n\geq1}$ and $(b_n)_{n\geq1}$, such that $a_i<b_i$ for every $i$ and $\{(a_n+a_m+1,b_n+b_m):n<m\}\cup\{(a_n,b_m):n<m\}$ is monochromatic.
\end{lemma}
\begin{proof}
Let $\Phi$ be the colouring in Lemma~\ref{Marialemma1}. Define $\Psi$ by $\Psi(a,b)=\Phi(a,b+1)$. Suppose we can find sequences $(a_n)_{n\geq 1}$ and $(b_n)_{\geq1}$ with the above properties. Let $d_n=b_n+1$. Then for $n<m$ we have $\Phi(a_n, d_m)=\Phi(a_n,b_m+1)=\Psi(a_n,b_m)$ and $\Phi(a_n+a_m,d_n+d_m)=\Phi(a_n+a_m,b_n+b_m+2)=\Psi(a_n+a_m,b_n+b_m+1)$, contradicting Lemma~\ref{Marialemma1}.
\end{proof}
The next lemma is proved in a completely analogous manner; we omit the proof.
\begin{lemma}
\label{Marialemma3}
There exists a finite colouring $\Psi'$ of $\mathbb N^{(2)}$ such that we cannot find two strictly increasing sequences of natural numbers, $(a_n)_{n\geq1}$ and $(b_n)_{n\geq1}$, such that $a_i<b_i$ for every $i\geq1$ and $\{(a_n+a_m-1, b_n+b_m):n<m\}\cup\{(a_n, b_m):n<m\}$ is monochromatic.\hfill$\square$\end{lemma}
\par Finally, we note a simple fact that we will use repeatedly. 
\begin{lemma}
\label{2colourslemma}
There exists a finite colouring $\varphi:\mathbb Z\to\{0,1\}$ such that $\varphi(k+1)\neq \varphi(2k)$ and
$\varphi(k+1)\neq \varphi(2k+1)$ for all $k\notin\{0,1\}$, and $\varphi(0)\neq\varphi(1)$ and $\varphi(2)\neq\varphi(3)$.
\end{lemma}
\begin{proof}
We build $\varphi$ inductively. Let $\varphi(0)=\varphi(2)=0$ and $\varphi(1)=\varphi(3)=1$. We now assume that $l\leq -1$, $k\geq 2$ and that $\varphi$ has been defined on $\{2l+2,2l+3,\cdots,2k-1\}$. Since
$0<k+1\leq 2k-1$, $\varphi(k+1)$ is defined, thus we set $\varphi(2k)=\varphi(2k+1)=1-\varphi(k+1)$. Similarly, since $2l+2\leq l+1\leq0$, $\varphi(l+1)$ is defined, so we set $\varphi(2l)=\varphi(2l+1)=1-\varphi(l+1)$, which finishes the induction step.
\end{proof}
\section{Colouring the naturals}

To illustrate the usefulness of Lemma~\ref{Marialemma1}, we use
it here to give a short proof of the result of Hindman \cite{H} 
about pairwise sums and products in the naturals. Because of the use of 
Lemma~\ref{Marialemma1}, what we are really doing is analysing the positions
of the digits in binary of the numbers that are themselves the positions of 
the digits in binary of the terms of the sequence.

For a natural number $a$, we write $e_2(a)$ for the end of $a$ (the subscript is
because later we will be looking at non-binary bases) and $s_2(a)$ for the start of
$a$. We also write $g_a$ for the difference between the positions of the two
most significant 1s of $a$ in binary, and call it the `gap' or `left gap' of $a$.
Thus for example $10001010$ has gap 4.

\begin{theorem}\label{theoremN}
There exists a finite colouring $\theta$ of $\mathbb N$ such that there is no injective sequence $(x_n)_{n\geq1}$ of natural numbers with the property that all numbers $x_n+x_m$ and $x_n x_m$ for all $1\leq n<m$ have the same colour.
\end{theorem}
\begin{proof} We begin by extending the colouring $\Phi$ from 
Lemma~\ref{Marialemma1} 
to a colouring of $(\mathbb N\cup\{0\})\times(\mathbb N\cup\{0\})$ by setting $\Phi(a,b)$ to be 0 if $a=0$ or $b=0$ or $a\geq b$. Now let $a$ be a natural number. We define $$\theta(a)=(p_a, e_2(a)\text{ mod }2,g_a\text{ mod }2,\Phi(e_2(a),s_2(a)),\Phi(e_2(a),s_2(a)+1),\varphi((e_2(a)),t_a)$$ where $p_a$ is 1 if $a$ is a power of 2 and 0 otherwise, and $t_a=0$ if $g_a=1$ and 1 otherwise. Observe that $\varphi$ ensures that there are no two numbers $a$ and $b$ of the same colour such that their end positions are $i+1$ and $2i$ respectively, for some $i\neq 1$.\par Suppose for a contradiction that there exists a strictly increasing sequence $(x_n)_{n\geq1}$ such that all pairwise sums and products have the same colour with respect to $\theta$. We observe that the first component of the colouring tells us that we cannot have two distinct powers of 2 in our sequence, and so we may assume that no term is a power of 2. Let $a_n$ be the position of the last digit of $x_n$ 
(i.e.~$a_n=e_2(x_n)$). Note
that the position of the last digit of $x_nx_m$ is $a_n+a_m$. Similarly, let $b_n$ be the position of the first digit of $x_n$ (i.e.~$b_n=s_2(x_n)$). We know that there will either be infinitely many $x_n$ such that $x_n<2^{b_n}\sqrt 2$, or infinitely many $x_n$ such that $x_n>2^{b_n}\sqrt 2$. By passing to a subsequence we may assume that either $x_n<2^{b_n}\sqrt 2$ for all $n$, or $x_n>2^{b_n}\sqrt2$ for all $n$. In the first case, the position of the first digit of $x_nx_m$ is $b_n+b_m$, while in the second case it is $b_n+b_m+1$.\par Assume first that all elements of the sequence end at position $1$. We either have infinitely many terms with the same gap, or infinitely many terms with pairwise distinct gaps. If the latter is true we may assume that $(x_n)_{n\geq1}$ has pairwise distinct gaps. Therefore we can find two $m$ and $n$ such that $x_n=2+2^i+\cdots$ and $x_{m}=2+2^{j}+\cdots$ where $2<i<j$. In this case the gap of the sum is $i-2$, while the gap of the product is $i-1$, a contradiction. Therefore we may assume that all $x_n$ end at position 1 and they have the same gap $g'$.\par If $g'>1$ then by the pigeonhole principle (and passing to a subsequence) we may assume that all terms have the same digit in position $g'+2$. Now it is easy to see that the sum of any two terms has gap $g'$, while the product has gap $g'+1$, a contradiction.
Hence we must have $g'=1$.
\par In other words, we may assume that all terms end $2+2^2+\cdots$, and by the pigeonhole principle we may further assume that the digit in position 3 is the same for all terms. A simple computation shows that the sum of any two terms has gap 1, while the product does not, a contradiction.\par This shows that we must have infinitely many terms that do not end at position $1$. Then,  by passing to a subsequence, we may assume that no term of the sequence ends at position $1$. If two terms $x_n$ and $x_m$ end at the same position, say $i\neq1$, then they cannot have the same gap. Indeed, if that were the case, the position of the last digit of $x_n+x_m$ is $i+1$, while the position of the last digit of $x_nx_m$ is $2i$, a contradiction. Thus we have $x_n=2^i+2^{i+k_1}+\cdots$ and $x_m=2^i+2^{i+k_2}\cdots$ for some $0<k_1<k_2$ (without loss of generality). The gap of the product is $k_1$. If $k_1\neq 1$ then the gap of the sum is $k_1-1$, a contradiction. But among any three terms that have the same end positions (and thus different gaps), we must always have two with gaps not equal to $1$. In other words, for any end position there are at most two terms that end there. By passing to a subsequence we may assume that the terms have right to left disjoint supports.
\par To sum up, by passing to a subsequence, we may assume that the terms $x_n$ are strictly increasing and have pairwise left to right disjoint supports. Thus the start and end positions form two increasing sequences, and since for $n<m$  we have $e_2(x_n+x_m)=a_n$ and $s_2(x_n+x_m)=b_m$, we are done by Lemma~\ref{Marialemma1} or Lemma~\ref{Marialemma2}.\end{proof}
\section{Colouring the reals}

In this section we prove the result about the reals mentioned in the
introduction, that there is a colouring for which no sequence that is bounded and bounded away from zero has all of its pairwise sums
and products monochromatic. There is a fair amount of notation,
which will also be used in later sections, but all of it is very
simple and self-explanatory. The aim is to analyse carefully how the `starting'
few 1s (in binary) of the numbers behave, and especially how close together
those first few 1s are.

\par For $x\in\mathbb R^+$, we define $a(x)$ to be the unique integer such that 
$2^{a(x)}\leq x<2^{a(x)+1}$. Moreover, for $x\in\mathbb R^+\setminus\{2^k:k\in\mathbb Z\}$,
we define $b(x)=a(x-2^{a(x)})$. In other words, for $x$ not an integer power of 2, $b(x)$ is the unique integer such that $2^{a(x)}+2^{b(x)}\leq x<2^{a(x)}+2^{b(x)+1}$. For $x\in\mathbb R^+\setminus\{2^k:k\in\mathbb Z\}$ we also define $c(x)$ to be the unique integer such that $2^{a(x)+1}-2^{c(x)+1}\leq x<2^{a(x)+1}-2^{c(x)}$.
\par Note that if $x\in\mathbb N$ then $a(x)$ is what we called the start of $x$ in Section 2 and Section 3. If $x$ is not a power of 2, then $b(x)$ is the position of the second most significant digit 1 in the base 2 expansion of $x$, and $c(x)$ is the position of the leftmost zero when $x$ is written in binary without leading $0$s. 
\par We now define $A_0=\{x\in\mathbb R^+:2^{a(x)}<x<2^{a(x)+\frac{1}{2}}\}$, $A_1=\{x\in\mathbb R^+:2^{a(x)+\frac{1}{2}}< x<2^{a(x)+1}\}$, $C_1=\{2^k:k\in\mathbb Z\}$ and $C_2=\{2^{k+\frac{1}{2}}:k\in\mathbb{Z}\}$. We observe that $A_0, A_1, C_1$, and $C_2$ are pairwise disjoint sets that partition $\mathbb R^+$, and $A_0$ and $A_1$ are open in $\mathbb R^+$, while $C_1$ and $C_2$ are countable.
\par Recalling the colouring $\varphi$ in Lemma~\ref{2colourslemma}, define $G_i=\{x\in\mathbb R^+\setminus C_1:\varphi(a(x))=i\}$ for $i\in\{0,1\}$. Since $G_i$ is the union of all the open intervals $(2^k,2^{k+1})$ where $k\in\mathbb Z$ and $\varphi(k)=i$, we see that $G_i$ is open in $\mathbb R^{+}$. Moreover, $C_1$, $G_0$ and $G_1$ also form a partition of the positive reals, where $C_1$ is countable and $G_0$ and $G_1$ are open.
\par Next we define $C_3=\{2^k+2^l:k,l\in\mathbb Z$ and $l<k\}$, and $H_i=\{x\in\mathbb R^+\setminus (C_1\cup C_3): a(x)-b(x)\equiv i\mod3\}$ for $i\in\{0,1,2\}$. By writing $H_i$ as the union of all open intervals $(2^k+2^l,2^k+2^{l+1})$ where $k,l\in\mathbb Z$, $l<k$ and $k-l\equiv i\mod 3$, we have that $H_i$ is open in $\mathbb R^+$ for $i\in\{0,1,2\}$. As before, $C_1$, $C_3$, $H_0$, $H_1$ and $H_2$ partition the positive reals.
\par Define now $C_4=\{2^k-2^l:k,l\in\mathbb Z$ and $l<k\}$, and $J_i=\{x\in\mathbb R^+\setminus C_4: a(x)-c(x)\equiv i\mod3\}$ for $i\in\{0,1,2\}$. Note that $C_1\subset C_4$ and $C_3\cap C_4=\{2^{k+1}+2^k:k\in\mathbb Z\}\neq\emptyset$. By writing $J_i$ as the union of all open intervals $(2^{k+1}-2^{l+1},2^{k+1}-2^{l})$ where $k,l\in\mathbb Z$, $l<k$ and $k-l\equiv i\mod 3$, we see that $J_i$ is open in $\mathbb R^+$ for $i\in\{0,1,2\}$. Also, $C_4$, $J_0$, $J_1$ and $J_2$ partition the positive reals.
\par Finally, we define $C_5=\{2^{k+1}(1-2^{l-k})^{\frac{1}{2}}:k,l\in\mathbb Z$ and $l<k\}$, and $B_i=\{x\in\mathbb R^+\setminus(C_1\cup C_5):x<2^{a(x)+1}(1-2^{c(x)-a(x)})^{\frac{1}{2}}$ and 
$a(x)-c(x)\equiv i\mod3$, or $x>2^{a(x)+1}(1-2^{c(x)-a(x)})^{\frac{1}{2}}$ and
$a(x)-c(x)\equiv i+1\mod3\}$ for $i\in\{0,1,2\}$. Note that $C_2\subset C_5$. Since $B_i$ can be written as the union of all the sets of the form $(2^{k+1}-2^{l+1}, 2^{k+1}(1-2^{l-k})^{\frac{1}{2}})$ where $l,k\in\mathbb Z$, $l<k$ and $k-l\equiv i\mod 3$, and all the sets of the form $(2^{k+1}(1-2^{l-k})^{\frac{1}{2}},2^{k+1}-2^l)$ where $k,l\in\mathbb Z$, $l<k$ and $k-l\equiv i+1\mod3$, we see that $B_i$ is open in $\mathbb R^+$ for all $i\in\{0,1,2\}$. Also, $C_1$, $C_5$, $B_0$, $B_1$ and $B_2$ partition the positive reals.
\par We are now ready to define our colouring $\nu$. To start with, we let $C_1$, $C_2$, $C_3\setminus C_4$, $C_4\setminus C_1$ and $C_5\setminus C_2$ be five colour classes of $\nu$. If $x\in\mathbb R^+\setminus(C_1\cup C_2\cup C_3\cup C_4\cup C_5)$, then we set $\nu(x)=(w_1,w_2,w_3,w_4,w_5)$, where $w_i=i$ if $x\in A_i$, $w_2=i$ if $x\in G_i$, $w_3=i$ if $x\in H_i$, $w_4=i$ if $x\in J_i$ and $w_5=i$ if $x\in B_i$.
\par It is important to note that, with the exception of the five countable classes defined first, the colour classes of $\nu$ are open (as a consequence of $C_1\cup\cdots\cup C_5$ being closed).
\begin{theorem}\label{0infinity}
Let $(x_n)_{n\geq1}$ be an injective sequence of positive reals with the property 
that all numbers $x_n+x_m$ and $x_n x_m$ for all $1\leq n<m$ have the same colour. 
Then $(x_n)_{n\geq1}$ cannot
be bounded and bounded away from zero.
\end{theorem}
\begin{proof}
The colour class of the pairwise sums and products of $(x_n)_{n\geq1}$ cannot be any of $C_1$, $C_2$, $C_3\setminus C_4$, $C_4\setminus C_1$ and $C_5\setminus C_2$. Indeed, the proofs for $C_1$ and $C_2$ are an easy exercise. The proofs for $C_3$ and $C_5$, while routine, are lengthy, and so are presented in the Appendix. The proof for $C_4$ is very similar to the one for $C_3$, and so we omit it. Therefore $x_n+x_m$ and $x_nx_m$ are all in $\mathbb R^+\setminus(C_1\cup C_2\cup C_3\cup C_4\cup C_5)$ for all $n<m$.
\par Suppose for a contradiction that $(x_n)_{n\geq1}$ is
bounded and bounded away from zero. This immediately implies that the sequence of integers $(a(x_n))_{n\geq1}$ is bounded. By passing to a subsequence, we may assume that $(a(x_n))_{n\geq1}$ is constant, and thus equal to some fixed integer $k$. Moreover, by the pigeonhole principle and passing to another subsequence, 
we may assume that either $x_n<2^{a(x_n)+\frac{1}{2}}$ for all $n$ or $2^{a(x_n)+\frac{1}{2}}\leq x_n$ for all $n$. 
\par Let $n$ and $m$ be two distinct natural numbers. Since $a(x_n)=a(x_m)=k$ we have that $2^{k+1}<x_n+x_m<2^{k+2}$
and $2^{2k}<x_nx_m<2^{2k+2}$. This implies that $a(x_n+x_m)=k+1$ and that either $a(x_nx_m)=2k$, or $a(x_nx_m)=2k+1$. Let $i\in\{0,1\}$ be such that $x_n+x_m\in G_i$ and $x_nx_m\in G_i$. In other words we must have $\varphi(a(x_n+x_m))=\varphi(a(x_nx_m))$, which implies that $\varphi(k+1)=\varphi(2k)$ or $\varphi(k+1)=\varphi(2k+1)$, and thus $k\in\{0,1\}$.
\par We consider first the case when $k=1$. This means that $2<a_n<4$ and $a(x_nx_m)=a(x_n+x_m)=2$ for all distinct naturals $n$ and $m$. Hence we must have $2<x_n<2^{\frac{3}{2}}$ for all $n$.
\par We first assume that the integer sequence $(b(x_n))_{n\geq1}$ is bounded. By passing to a subsequence, we may assume that $(b(x_n))_{n\geq1}$ is constant and equal to a fixed integer $l<k=1$. Since $x_n\geq 2^{a(x_n)}+2^{b(x_n)}$ for all $n$, we cannot have $l=0$, or else $x_n\geq 2+1=3>2^{\frac{3}{2}}$, and so $l\leq-1$.
\par Let $m$ and $n$ be two distinct natural numbers. By the above we have that $x_n=2+2^l+u$ and $x_m=2+2^l+v$ for some $0\leq u,v<2^l$. Next we have that $x_n+x_m=4+2^{l+1}+u+v$ and $0\leq u+v<2^{l+1}$, thus $b(x_n+x_m)=l+1$, and consequently $a(x_n+x_m)-b(x_n+x_m)=2-(l+1)=1-l$.
\par On the other hand, $x_nx_m=4+2^{l+2}+(2^l+2)(u+v)+uv+2^{2l}$. The sum of terms involving the variables $u$ and $v$ can be bounded as follows: $(2^l+2)(u+v)+uv+2^{2l}<(2^l+2)2^{l+1}+2^{2l}+2^{2l}
=2^{2l+2}+2^{l+2}$. Therefore we trivially have $4+2^{l+2}<x_nx_m$ and $x_nx_m<4+2^{l+2}+2^{2l+2}+2^{l+2}
=4+2^{l+3}+2^{2l+2}<4+2^{l+4}$. This tells us that either $b(x_nx_m)=l+2$, or $b(x_nx_m)=l+3$, thus either
$a(x_nx_m)-b(x_nx_m)=-l$, or $a(x_nx_m)-b(x_nx_m)=-l-1$. In both cases $a(x_nx_m)-b(x_nx_m)$ and $a(x_n+x_m)-b(x_n+x_m)$ are not congruent $\mod3$, a contradiction.
\par Therefore we must have that $(b(x_n))_{n\geq1}$ is unbounded and, by passing to a subsequence, we may assume that $(b(x_n))_{n\geq1}$ is strictly decreasing.
\par Let $n$ be a natural number and $l=b(x_n)$. We know that there exists $u$ such that $0\leq u<2^l$ and $x_n=2+2^l+u$. We now pick an integer $s<l$ such that $u+2^s<2^l$, and then a natural number $m$ such that $b(x_m)<s$. Let $t=b(x_m)$ and $x_m=2+2^t+v$, where $0\leq v<2^t$. It follows that $x_n+x_m=4+2^l+u+2^t+v$. By all the above we have that $u+2^t+v<u+2^{t+1}\leq u+2^s<2^l$. Thus $b(x_n+x_m)=l$ and $a(x_n+x_m)-b(x_n+x_m)=2-l$.
\par Finally, since $2+2^l\leq x_n<2+2^{l+1}$ and $2+2^t\leq x_m<2+2^{t+1}$, we first have that 
$4+2^{l+1}<4+2^{l+1}+2^{t+1}+2^{l+t}\leq x_nx_m$. Moreover, $x_nx_m<4+2^{l+2}+2^{t+2}+2^{l+t+2}<4+2^{l+3}$. Putting these together we see that either $b(x_nx_m)=l+1$ or
$b(x_nx_m)=l+2$. Thus either $a(x_nx_m)-b(x_nx_m)=1-l$, or $a(x_nx_m)-b(x_nx_m)=-l$, neither of which is congruent to $a(x_n+x_m)-b(x_n+x_m)\mod3$, a contradiction. 
This concludes the case when $k=1$. 
\par We must therefore have $k=0$. In other words $a(x_n)=0$, $2^{\frac{1}{2}}\leq x_n<2$, and $a(x_n+x_m)=a(x_nx_m)=1$ for all distinct natural numbers $n$ and $m$. Since there is at most one $n$ such that $x_n=2^{\frac{1}{2}}$, by passing to a subsequence we may assume that $2^{\frac{1}{2}}<x_n<2$ for all $n$.
\par We observe that if $2^{\frac{1}{2}}<x_n<\frac{3}{2}$ and $2^{\frac{1}{2}}<x_n<\frac{3}{2}$ for two distinct $m$ and $n$, then $2\cdot2^{\frac{1}{2}}=2^{\frac{3}{2}}<x_n+x_m<3$, thus 
$x_n+x_m\in A_1$, while $2<x_nx_m<9/4<2^{\frac{3}{2}}$, so $x_nx_m\in A_0$, a contradiction. Therefore, by passing to a subsequence, we may assume that $\frac{3}{2}\leq x_n<2$. This immediately implies that $x_n\geq 2^1-2^{-1}=2^{a(x_n)+1}-2^{-2+1}$, and so $c(x_n)\leq -2$ for all $n$.
\par We first assume that the integer sequence $(c(x_n))_{n\geq1}$ is bounded. Thus by passing to a subsequence we may assume that it is constant and equal to a fixed integer $l\leq -2$. Let $m$ and $n$ be two distinct natural numbers. Then we have that $2-2^{l+1}\leq x_n<2-2^l$ and $2-2^{l+1}\leq x_m<2-2^l$. Summing the above we obtain $4-2^{l+2}\leq x_n+x_m<4-2^{l+1}$, and thus $c(x_n+x_m)=l+1$ and consequently $a(x_n+x_m)-c(x_n+x_m)=-l$. On the other hand, multiplying the above gives $4-2^{l+3}+2^{2l+2}\leq x_nx_m<4-2^{l+2}+2^{2l}$. The lower bound is trivially greater than $4-2^{l+3}$, and $2^{l+2}-2^{2l}>2^{l+1}$, so $4-2^{l+2}+2^{2l}<4-2^{l+1}$. This means that $c(x_nx_m)$ is either $l+1$ or $l+2$. Since $c(x_nx_m)=l+2$ implies $a(x_nx_m)-c(x_nx_m)=-l-1$ which is not congruent to $-l=a(x_n+x_m)-c(x_n+x_m)\mod3$, we conclude that $c(x_nx_m)=l+1$ for all $n\neq m$, which can be written as $4-2^{l+2}\leq x_nx_m<4-2^{l+1}$ for all $n\neq m$.
\par Observe that if $x_n<2(1-2^l)^{\frac{1}{2}}$ and $x_m<2(2-2^l)^{\frac{1}{2}}$ for two distinct positive integers $m$ and $n$, then $x_nx_m<4(1-2^l)=4-2^{l+2}$, which contradicts $c(x_nx_m)=l+1$. Therefore, by passing to a subsequence, we may assume that $x_n\geq 2(1-2^l)^{\frac{1}{2}}$ for all $n$.\par Let $n\neq m$ be two natural numbers. Then $x_n+x_m\geq 4(1-2^l)^{\frac{1}{2}}=
4(1-2^{c(x_n+x_m)-a(x_n+x_m)})^{\frac{1}{2}}$. Let $i\in\{0,1,2\}$ such that $-l=a(x_n+x_m)-c(x_n+x_m)\equiv i+1\mod3$. This means that
$x_n+x_m\in B_i$, and consequently $x_nx_m\in B_i$. On the other hand, since $x_n<2-2^l$ and $x_m<2-2^l$, it is easy to check that the product $x_nx_m<4-2^{l+2}+2^{2l}=4(1-2^l+2^{2l-2})<4(1-2^l)^{\frac{1}{2}}$. Since $a(x_nx_m)-c(x_nx_m)=1-(l+1)=-l$ we have that $x_nx_m<4(1-2^{c(x_nx_m)-a(x_nx_m)})^{\frac{1}{2}}$, and thus $x_nx_m\in B_j$ where $j\in\{0,1,2\}$ and $j\equiv-l\mod3\equiv i+1\mod 3$. But this is a contradiction since it implies that $i\neq j$, so that the sum and the product are in different $B$-classes.
\par Therefore we must have that the sequence $(c(x_n))_{n\geq1}$ is unbounded and, by passing to a subsequence, we may assume that it is strictly decreasing.
\par Let us first assume that there exist $n<m$ such that $x_n=2-2^{c(x_n)+1}$ and $x_m=2-2^{c(x_m)+1}$. Then we have that $x_n+x_m=4-2^{c(x_n)+1}-2^{c(x_m)+1}$, and since $c(m)<c(n)$ we get that $4-2^{c(x_n)+2}<x_n+x_m<4-2^{c(x_n)+1}$, so $c(x_n+x_m)=c(x_n)+1$ and consequently $a(x_n+x_m)-c(x_n+x_m)=-c(x_n)$. On the other hand, $x_nx_m=4-2^{c(x_n)+2}-2^{c(x_m)+2}+2^{c(x_n)+c(x_m)+2}
=4-2^{c(x_n)+2}-2^{c(x_m)+2}(1-2^{c(x_n)})$. Hence we have that $4-2^{c(x_n)+3}<4-2^{c(x_n)+2}-2^{c(x_m)+2}<x_nx_m<4-2^{c(x_n)+2}$. It follows that $c(x_nx_m)=c(x_n)+2$, so$a(x_nx_m)-c(x_nx_m)=-c(x_n)-1$, a contradiction.
\par Finally, after passing to a subsequence, we may assume that for every $n$ there exists $u_n$ such that $0<u_n<2^{c(x_n)}$ and $x_n=2-2^{c(x_n)+1}+u_n$. Let $n$ be a natural number and let $s\in\mathbb Z$ be such that $u_n+2^s<2^{c(x_n)}$. Since the sequence $(c(x_n))_{n\geq1}$ is strictly decreasing and unbounded, we can find $m>n$ such that $c(x_m)<\min\{s,\log_2 u_n-1\}$. It then follows that $x_n+x_m=4-2^{c(x_n)+1}+u_n-2^{c(x_m)+1}+u_m$. We observe that $0<u_n-2^{c(x_m)+1}+u_m<u_n-2^{c(x_m)+1}+2^{c(x_m)}<u_n+2^{c(x_m)}
<u_n+2^s<2^{c(x_n)}$.  This means that $4-2^{c(x_n)+1}<x_n+x_m<4-2^{c(x_n)+1}+2^{c(x_n)}=4-2^{c(x_n)}$, and so  $c(x_n+x_m)=c(x_n)$ and consequently $a(x_n+x_m)-c(x_n+x_m)=1-c(x_n)$.
\par We are now going to analyse the product $x_nx_m$. We have that $2-2^{c(x_n)+1}<x_n<2-2^{c(x_n)}$ and
$2-2^{c(x_m)+1}<x_m<2-2^{c(x_m)}$. By multiplying the above inequalities we obtain that $4-2^{c(x_n)+2}-2^{c(x_m)+2}+2^{c(x_n)+c(x_m)+2}<x_nx_m$ and $x_nx_m<
4-2^{c(x_n)+1}-2^{c(x_m)+1}+2^{c(x_n)+c(x_m)}$. We consider these two inequalities separately.
\par First we have that $4-2^{c(x_n)+1}-2^{c(x_m)+1}+2^{c(x_n)+c(x_m)}=
4-2^{c(x_n)+1}-2^{c(x_m)}(2-2^{c(x_n)})<4-2^{c(x_n)+1}$, thus $x_nx_m<4-2^{c(x_n)+1}$.
\par Secondly we have that $4-2^{c(x_n)+2}-2^{c(x_m)+2}+2^{c(x_n)+c(x_m)+2}
>4-2^{c(x_n)+2}-2^{c(x_m)+2}>4-2^{c(x_n)+3}$, since $c(x_m)<c(x_n)$.
\par Putting everything together we get that $4-2^{c(x_n)+3}<x_nx_m<4-2^{c(x_n)+1}$, tand hus either $c(x_nx_m)=c(x_n)+1$ or $c(x_nx_m)=c(x_n)+2$. This means that either $a(x_nx_m)-c(x_nx_m)=-c(x_n)$,
or  $a(x_nx_m)-c(x_nx_m)=-c(x_n)-1$, neither of which is congruent to $a(x_n+x_m)-c(x_n+x_m)=1-c(x_n)\mod3$, a contradiction.
\end{proof}

It is important to point out that the colouring $\nu$ cannot be used to 
rule out similar statements about sums and products from a sequence $(x_n)_{n \geq 1}$ 
that tends to zero. Indeed, since each colour class of $\nu$ is measurable (being
either countable or open), the result of \cite{BHL} tells us that there is a
sequence with all of its products and all of its sums (even infinite sums) having
the same colour for $\nu$. 

\section{Combining an extension of $\theta$ over the rationals with $\nu$}
In this section we will build a colouring of the positive rationals via an `extension' of the colouring $\theta$ that also incorporates $\nu$. This colouring will force any bounded sequence with monochromatic pairwise sums and products to have the set of primes which divide the denominators of the terms of the sequence 
to be infinite. 

Roughly speaking, we will be concerned with how a number ends, not just how it
starts, and therefore we will be considering numbers written not in binary (of
course) but rather in the smallest base for which they terminate. The analysis is
considerably more complicated than it would be for binary. There is also the
issue that different numbers will have different `smallest bases', but it turns
out that this will not cause too much of a problem.
\par Let $(p_n)_{n\geq1}$ be the enumeration
of all primes in increasing order, and $P_n=\displaystyle{\prod_{k=1}^n p_k}$ for all $n\in\mathbb{N}$. Let also $T_n=\mathbb Q_{(n)}\cap(0,1)$. In other words, $T_n$ consists of all the rationals between 0 and 1 for which, in reduced form, the denominator does not have any $p_t$ with $t>n$ as a factor. For completeness, define $T_0=\emptyset$. If $x \in T_n \setminus T_{n-1}$ we may say that $P_n$ is the
`minimal base' of $x$.
\par For $n\in\mathbb N$ and $x\in T_n$, we define $s_n(x)$ to be the position
of the leftmost significant digit and $e_n(x)$ the position of the rightmost significant
digit in the base $P_n$ expansion of $x$. For example, if $x$ has the base 6 
expansion $405.00213$ then $s_3(x)=2$ and $e_3(x)=-5$. For $x\in\mathbb N$, so
that $e_2(x)$ and $s_2(x)$ are the position of the rightmost significant digit and leftmost significant digit respectively in the binary expansion of $x$, we set $d(x)$ to be the digit in position $e_2(x)+1$. Finally, for $x,y\in\mathbb N$, define $g(x,y)=0$ if $e_2(y)>s_2(x)$ and
$g(x,y)=1$ if $e_2(y)\leq s_2(x)$.
\par The colouring $\Phi$ of $\mathbb N^{(2)}$ defined previously can be rewritten as follows: $\Phi(x,y)=(e_2(x)\mod2,e_2(y)\mod2,d(x),d(y),g(x,y))$. We also define the colouring $\Psi'$ of $\mathbb N^{(2)}$, which is very similar in spirit to the previously defined colouring $\Psi$, by $\Psi'(x,y)=\Phi(1,2)$ if 
$x=1$, and $\Psi'(x,y)=\Phi(x-1,y)$ if $x>1$.
\par We are now ready to define a colouring $\mu$ of $\mathbb Q$ as follows. If $x\geq1$, let $\mu(x)=\nu(x)$. Otherwise, for any $x\in\mathbb{Q}\cap(0,1)$, there exists a unique $n\in\mathbb N$ such that $x\in T_n\setminus T_{n-1}$. Consequently, define $$\mu(x)=(\nu(x), \Phi(-s_n(x),-e_n(x)),\Psi'(-s_n(x),-e_n(x))).$$
The following is what we wish to prove.

\begin{theorem}\label{wasunlabelled}
Let $(x_n)_{n\geq1}$ be a bounded sequence of positive rationals such that the set $\{x_n+x_m, x_nx_m:n\neq m\}$ is monochromatic with respect to $\mu$. Then for any $k\in\mathbb N$ there exist $l$ and $n$ such that $x_n\in T_l\setminus T_k$.
\end{theorem}
\begin{proof}
Because the sequence $(x_n)_{n\geq1}$ is monochromatic with respect to $\mu$ it is also monochromatic with respect to $\nu$. Since $(x_n)_{n\geq1}$ is bounded, Theorem~\ref{0infinity} tells us that $(x_n)_{n\geq1}$ must converge to $0$, and so we may assume that all terms are less than 1.\par Assume for a contradiction that there exists $k\in\mathbb N$ such that $x_n\in T_k$ for all $n\in\mathbb N$. By passing to a subsequence, we can assume that for all $n$ $x_n\in T_{t}\setminus T_{t-1}$ for some $t\leq k$. In other words, the minimal base of the form $P_s$ for $x_n$ is $P_t$, for all $n\geq1$. Since $(x_n)_{n\geq1}$ converges to 0, $s_t(x_n)$ and $e_t(x_n)$ must 
tend to $-\infty$. In particular, we may assume from now on that $s_t(x_n)<-1$ for all $n$. Moreover, by passing to a subsequence, we may assume that the sequence is strictly decreasing and that all of its terms have pairwise left to right disjoint support -- in other words, if $n<m$ then $e_t(x_n)>s_t(x_m)$. Also, by the pigeonhole principle, there exists a subsequence for which all terms have the same last digit, say 
$0<d<P_t$, and by passing to that subsequence we may assume that this is the case for  $(x_n)_{n\geq1}$ itself.\par Let $n<m$ be positive integers. Then, because $x_n$ and $x_m$ have disjoint supports in base $P_t$, which is their minimal base,  $x_n+x_m$ also has minimal base $P_t$. Furthermore, $s_t(x_n+x_m)=s_t(x_n)$ and $e_t(x_n+x_m)=e_t(x_m)$. It is also easy to see that if both $x_n$ and $x_m$ have minimal base $P_t$ then so does $x_nx_m$.
\par We note that if $x\in T_t\setminus T_{t-1}$, then $-e_t(x)$ is the smallest positive integer $u$ such that $x(P_t)^u\in\mathbb N$. Clearly $x_nx_m (P_t)^{-e_t(x_n)-e_t(x_m)}\in\mathbb N$, and thus $e_t(x_nx_m)\geq e_t(x_n)+e_t(x_m)$.
\par Now suppose that there exists $k'\in\mathbb N$ smaller than $-e_t(x_n)-e_t(x_m)$ such that $x_nx_m (P_t)^{k'}\in\mathbb N$. It follows that $x_n(P_t)^{-e_t(x_n)}x_m(P_t)^{-e_t(x_m)}(P_t)^{k'+e_t(x_n)+e_t(x_m)}\in\mathbb N$. But $x_n(P_t)^{-e_t(x_n)}\equiv x_m(P_t)^{-e_t(x_m)}\equiv d\mod P_t$. Because the power of $P_t$ is negative, we must have that $P_t$ divides $d^2$, and since $P_t$ is a product of distinct primes, we must in fact have that $P_t$ divides $d$, a contradiction. Therefore $e_t(x_nx_m)=e_t(x_n)+e_t(x_m)$.
\par Finally, for $x\in T_t\setminus T_{t-1}$, $s_t(x)$ is the unique integer $l$ such that $(P_t)^{l+1}>x\geq(P_t)^l$. By the pigeonhole principle we either have $x_n\geq\sqrt{P_t}(P_t)^{s_t(x_n)}$ for infinitely many $n$ or $x_n<\sqrt{P_t}(P_t)^{s_t(x_n)}$ for infinitely many $n$. By passing to a subsequence we may assume that we are either in the first case for all $n$ or in the second case for all $n$. In the first case $s_t(x_nx_m)=s_t(x_n)+s_t(x_m)+1$ for all $m\neq n$, while in the second case $s_t(x_nx_m)=s_t(x_n)+s_t(x_m)$.\par Let $a_n=-s_t(x_n)>1$ and $b_n=-e_t(x_n)>a_n$ for all $n\in\mathbb N$. Note that both $(a_n)_{n\geq1}$ and $(b_n)_{n\geq1}$ are strictly increasing sequences of natural numbers. Then $\mu$ tells us that either $\Phi(a_n, b_m)=\Phi(a_n+a_m, b_n+b_m)$ for all $n<m$, or $\Phi(a_n-1, b_m)=\Phi(a_n+a_m-2, b_n+b_m)$ for all $n<m$, which contradicts Lemma~\ref{Marialemma1} or Lemma~\ref{Marialemma3}.
\end{proof}
\section{Exploring $\mu$ further}
It turns out that for $\mu$ we can find an injective sequence with all pairwise 
sums and products monochromatic, and actually even all finite sums and
products monochromatic. This shows that neither $\theta$ nor $\nu$ nor their product  can provide a counterexample for the 
`finite sums and products' problem in the set of all rationals. We include this result just out of interest; the reader can skip this section if desired.\par We say that the sequence $(y_n)_{n\geq1}$ is a \textit{product subsystem} of the sequence $(x_n)_{n\geq1}$ if there exists a sequence $(H_n)_{n\geq1}$ of finite sets of natural numbers such that for every $n\geq1$, $\max H_n <\min H_{n+1}$ and $y_n=\displaystyle\prod_{t\in H_n}x_t$.
\begin{theorem}\label{theoremextension}There exists a sequence $(y_n)_{n\geq1}$ in $\mathbb Q\cap (0,1)$ such that all of its finite sums and finite products are monochromatic with respect to $\mu$.
\end{theorem}
\begin{proof} Starting with $r_1=2$, we may inductively choose an increasing sequence 
$ (r_n)_{n\geq1}$ of natural numbers such that for all $n\in\mathbb N$ we have  $\displaystyle\sum_{i=1}^n{1\over p_{r_i}}<1$. By the Finite Sums Theorem
(or rather a simple corollary of it -- see Corollary 5.15 in \cite{HS}) we can choose a product subsystem $(x_n)_{n\geq 1}$ of $\left(\dfrac{1}{p_{r_i}}\right)_{n\geq1}$
such that all finite products of $(x_n)_{n\geq1}$ are monochromatic with respect to $\nu$ -- in other words, they are all members of a colour class of $\nu$, say $U$. 
\par The colouring $\nu$ of 
$\mathbb R^+$ consists of five countable classes and several classes that are open in $\mathbb R^+$. Recall that the countable colour classes are $C_1=\{2^k:k\in\mathbb Z\}$, $C_2=\{2^{k+\frac{1}{2}}:k\in\mathbb Z\}$, $C_3\setminus C_4=\{2^k+2^l:k,l\in\mathbb Z$ and $l<k\}\setminus C_4$,
$C_4\setminus C_1=\{2^k-2^l:k,l\in\mathbb Z$ and $l<k\}\setminus C_1$, and $C_5\setminus C_2=\{2^{k+1}(1-2^{l-k})^{\frac{1}{2}}:k,l\in\mathbb Z$ and $l<k\}\setminus C_2$. It is easy to see that $C_2$ contains only irrational numbers. Observe also that $C_5$ contains only irrational numbers, because $\left(1-\frac{1}{2^n}\right)^{\frac{1}{2}}$ is irrational for any $n\in\mathbb N$. (Indeed, suppose $\left(1-\frac{1}{2^n}\right)^{\frac{1}{2}}=\frac{p}{q}$ for some coprime $p,q\in\mathbb N$; we then get that $\frac{2^n-1}{2^n}=\frac{p^2}{q^2}$, so $2^n$ and $2^n-1$ have to be perfect squares, but no two perfect squares in $\mathbb N$ differ by 1.)
\par The classes $C_1$, $C_3\setminus C_4$, and $C_4\setminus C_1$ consist of rational number that have denominator (in reduced form) a power of 2, and thus none of them can be  in $U$ as no $x_n$ has this property since $r_1=2$. Furthermore, $C_2$ and $C_5\setminus C_2$ consist of irrational numbers, so $C_2\neq U$ and $C_5\setminus C_2\neq U$. We conclude that $U$ is an open colour class of $\nu$ that contains all the finite products of $(x_n)_{n\geq1}$.
\par We are now going to find a subsequence $(y_n)_{n\geq1}$ of $(x_n)_{n\geq1}$ such that all its finite sums are in $U$ as well. We proceed by induction. Let $y_1=x_1$. Now assume $n\geq1$ and that we have chosen
$y_1>y_2>\cdots>y_n$ such that $y_i\in\{x_j:j\in\mathbb N\}$ for all $1\leq i\leq n$, and that for any finite non-empty set $A$ of $\{1,2,\cdots,n\}$ we have $\displaystyle\sum_{i\in A} y_i\in U$.\par Because $U$ is open in $\mathbb R^{+}$, we can pick $\epsilon_A>0$ such that
$\left(\displaystyle\sum_{i\in A} y_i,\displaystyle\sum_{i\in A} y_i+\epsilon_A\right)\subset U$ for any finite non-empty set $A$ of $\{1,2,\cdots,n\}$. Let
$\epsilon=\min\{\epsilon_A, y_i:\emptyset\neq F\subseteq\{1,2,\cdots,n\}, 1\leq i\leq n\}$. Pick $m$ such that for all $j\geq m$ we have $x_j<\epsilon$, and set $y_{n+1}=x_m$. This finishes the induction step. Therefore, all the finite sums and all the finite products of the sequence $(y_n)_{n\geq1}$ are in $U$, and so are monochromatic for the 
colouring $\nu$.
\par To complete the proof we show that if $z$ is either a finite sum or a finite product of $(y_n)_{n\geq1}$ and $z\in T_k\setminus T_{k-1}$, then $e_k(z)=-1$, and consequently $s_k(z)=-1$.
\par First assume that $z$ is a finite product of elements of $(y_n)_{n\geq1}$. This implies that $z$ is a finite product of elements of $\left(\dfrac{1}{p_n}\right)_{n\geq1}$. Therefore there exists a finite set $A$ of natural numbers such that $z=\displaystyle\prod_{i\in A}\dfrac{1}{p_i}$, and thus $z\in T_k\setminus T_{k-1}$, where $k=\max A$. We observe that $zP_k=\displaystyle\prod_{i\in\{1,2,\ldots,k\}\setminus A}p_i<P_k$, so $z=\dfrac{z'}{P_k}$ for some $1\leq z'<P_k$, which implies that $e_k(z)=s_k(z)=-1$.
\par Finally, let $z=\displaystyle\sum_{i\in A}y_i$ for some finite set $A=\{j_1,j_2,\cdots,j_s\}$ of natural numbers of size $s>1$, where $j_1<j_2<\cdots<j_s$. Since $(y_n)_{n\geq1}$ is a subsequence of $(x_n)_{n\geq1}$, which is a 
product subsystem of $\left(\dfrac{1} {p_{r_n}}\right)_{n\geq1}$, for each $i\in\{1,2,\ldots,s\}$ there exists a finite set $F_i$ of natural numbers such that $\max F_i<\min F_{i+1}$ if
$i<s$, and $y_{j_i}=\displaystyle\prod_{t\in F_i}\dfrac{1}{p_{r_t}}$. Denote by $m_i$ the maximum of $F_i$ for all $i\in\{1,2,\ldots,s\}$, and let $k=r_{m_s}$, so that $z\in T_k\setminus T_{k-1}$.
\par We first note that $\displaystyle\sum_{i=1}^s\dfrac{1}{p_{r_{m_i}}}<1$, and thus $\displaystyle\sum_{i=1}^s{\dfrac{p_k}{p_{r_{m_i}}}}<p_k$. We now see that $zP_k=\left(\displaystyle\sum_{i=1}^s y_{j_i}\right)\displaystyle\prod_{m=1}^kp_m
=\left(\displaystyle\sum_{i=1}^s\prod_{t\in F_i}\dfrac{1}{p_{r_t}}\right)\displaystyle\prod_{m=1}^kp_m\leq\left(\displaystyle\sum_{i=1}^s\dfrac{1} {p_{r_{m_i}}}\right)\displaystyle\prod_{m=1}^kp_m=\left(\sum_{i=1}^s\dfrac{p_k}{p_{r_{m_i}}}\right)\displaystyle\prod_{m=1}^{k-1}p_m<P_k$, by the above observation. Therefore, as before, $z=\dfrac{z''}{P_k}$ for some $1\leq z''<P_k$, which implies $e_k(z)=s_k(z)=-1$.
\end{proof}
\section{Unbounded sequences in the rationals}
In this section we give a finite colouring of the rationals such
that no unbounded sequence whose denominators contain only
finitely many primes can have the set of all its finite sums
and products monochromatic.

The general aim is to write numbers as an integer part (which will be considered
in binary) and a fractional part (which will be considered in the `minimal base'
as in Section 5), although actually we will also make use of the integer part
written in that minimal base of the fractional part. By using the finite sums,
we hope to show that the `centres clear out', meaning that the fractional parts tend
to 0 (or 1) and the integer parts have ends that tend to infinity. This will then give
us the disjointness of support that we need to apply results conceptually similar to  
Lemma~\ref{Marialemma1}. For example, if the fractional parts tend to
0 and the integer parts have ends that tend to infinity then we will consider 
the relationship between quantities like the left gap of the integer part and 
the end of the fractional part -- the key point being that we will be able to
control how the integer parts behave under sum and product, because the 
fractional parts will be `too small to interfere'. 

\begin{theorem}\label{finitesums}
There exists a finite colouring $\alpha$ of the positive rationals such that there exists no unbounded sequence $(x_n)_{n\geq1}$ that has the set of all its finite sums and products monochromatic with respect to $\alpha$, with the set of primes that divide the denominators of its terms being finite.
\end{theorem}
\begin{proof} Let $S_n=\{x\in\mathbb Q^+: x$ has a terminating base $P_n$ expansion$\}$ for all $n>0$, and $S_0=\emptyset$. We first define the colouring $\alpha'$ of $\mathbb Q^+\setminus(\mathbb N\cup\{2^k:k\in\mathbb Z\}\cup(0,2])$ as follows: for $x\in S_r
\setminus S_{r-1}$ we set \begin{center}$\alpha'(x)=(a(x)\text{ mod } 2,a(\Frac(x))\text{ mod }2, \epsilon(\Frac(x))\text{ mod }2, e_r(\lfloor x\rfloor)\text{ mod }2, e_2(\lfloor x\rfloor)\text{ mod }2,\newline e_r(\lfloor x\rfloor+1)\text{ mod }2,e_2(\lfloor x\rfloor+1)\text{ mod }2,a(r(x))\text{ mod }3,p(x),q(x), q'(x),s(x),s'(x))$,\end{center} where $1-2^{\epsilon(\Frac(x)})\leq \Frac(x)<1-2^{\epsilon(\Frac(x))-1}$,  and as before $e_r(x)$ is the position of the rightmost significant digit in base $P_r$ and $e_2(x)$ is the position of the rightmost significant digit in binary,  and also $r(x)=\dfrac{x-2^{a(x)}}{2^{a(x)}}$, $p(x)$ is 0 if $\lfloor x\rfloor$ is a power of 2 and 1 otherwise, $q(x)$ is 0 if $a(x)-b(x)>e_r(\lfloor x\rfloor)$ and 1 otherwise, $q'(x)$ is 0 if $a(x)-b(x)>e_r(\lfloor x\rfloor+1)$ and 1 otherwise, $s(x)$ is 0 if $a(x)-c(x)>e_r(\lfloor x\rfloor)$ and 1 otherwise, $s'(x)$ is 0 if $a(x)-c(x)>e_r(\lfloor x\rfloor+1)$ and 1 otherwise. Here $\lfloor x\rfloor$ and $\Frac(x)$ are the integer and the fractional parts of $x$ respectively. 
\par We are now ready to define the colouring $\alpha$. Let $x\in\mathbb Q^+$. Then $\alpha(x)=(0,\theta(x))$ if $x\in\mathbb N$, $\alpha(x)=1$ if $x\in\{2^k:k\in\mathbb Z,k<0\}$, $\alpha(x)=2$ if $x\leq2$, $x\notin\mathbb N$ and $x\notin\{2^k:k\in\mathbb Z, k<0\}$, and $\alpha(x)=(1,\alpha'(x))$ otherwise. 
\par Suppose for a contradiction that a sequence as specified in the statement of the theorem exists. Since it is unbounded, we may assume that all its terms are greater than 2. Since $\theta$ prevents any sequence of natural numbers from having monochromatic pairwise sums and products, we may assume, by passing to a subsequence, that none of the $x_n$ are natural numbers -- and hence, since the set of the finite sums and products is monochromatic, also no finite sum or product of the $x_n$ is a natural number. Moreover, by looking at sums of two terms, it is easy to see that $p$ prevents the integer parts from being powers of 2, and thus we can assume that no $x_n$ has its integer part a power of 2. By assumption, and after passing to a subsequence, we may assume that there exists $r\in\mathbb N$ such that $x_n\in S_r\setminus S_{r-1}$ for all $n$. Since $S_r\setminus S_{r-1}$ is closed under multiplication, all the finite products are in $S_r\setminus S_{r-1}$ too.
\par Let $x_n=y_n+z_n$, where $y_n\in\mathbb N$ is the integer part of $x_n$ and $0<z_n<1$ is its fractional part. By passing to a subsequence we may assume that the sequence $(y_n)_{n\geq1}$ is strictly increasing and tending to infinity. Suppose that the sequence $(z_n)_{n\geq1}$ is bounded away from both 0 and 1, which is equivalent to saying that $a(z_n)$ and $\epsilon(z_n)$ are both bounded. Therefore, by passing to a subsequence, we may assume that there exist fixed integers $k<0$ and $l<1$ such that $a(z_n)=k$ and $\epsilon(z_n)=l$ for all $n$. We either have $z_n<\frac{1}{2}$ for infinitely many $n$ or $z_n\geq\frac{1}{2}$ for infinitely many $n$. 
\par In the first case, if $z_n$ and $z_m$ are less than $\frac{1}{2}$ then $\Frac(x_n+x_m)=z_n+z_m$, and thus $a(\Frac(x_n+x_m))=k+1\neq a(\Frac(x_n))\mod2$, a contradiction. In the second case, if $z_n$ and $z_m$ are at least $\frac{1}{2}$ then $\Frac(x_n+x_m)=z_n+z_m-1$, so that $1-\Frac(x_n+x_m)=1-z_n+1-z_m$ which implies that $\epsilon(\Frac(x_n+x_m))=l+1\neq \epsilon(\Frac(x_n))\mod2$, a contradiction. This tells us that, by passing to a subsequence, we may either assume that $z_n$ converges to 0 or that it converges to 1.
\par By passing to a subsequence we may assume that either $x_n<2^{a(x_n)+\frac{1}{2}}$ for all $n$ or 
$x_n\geq 2^{a(x_n)+\frac{1}{2}}$ for all $n$. In the first case $a(x_nx_m)=a(x_n)+a(x_m)$, while in the 
second case $a(x_nx_m)=a(x_n)+a(x_m)+1$ (for all $n\neq m$). Since, for $x\in\mathbb R^+\setminus(\mathbb N\cup C_1)$, 
$r(x)$ is the unique number strictly between 0 and 1 such that $x=2^{a(x)}(1+r(x))$, a simple computation shows 
that in the first case $r(x_nx_m)=r(x_n)+r(x_m)+r(x_n)r(x_m)$, while in the second case 
$r(x_nx_m)=\dfrac{r(x_n)+r(x_m)+r(x_n)r(x_m)-1}{2}$ for all $n\neq m$.
\par Suppose that $x_n<2^{a(x_n)+{\frac{1}{2}}}$ for all $n$ and that $r(x_n)$ is bounded away from $0$. Then $a(r(x_n))$ is bounded, so by  passing to a subsequence we may assume that there is an integer $l<-1$ such that $a(r(x_n))=l$ for all $n$ (Recall that we are in the case where $r(x_n)+r(x_m)+r(x_n)r(x_m)<1$ and thus $a(r(x_n))<-1$). Since $2^l\leq r(x_n)<2^{l+1}$ and $2^l\leq r(x_m)<2^{l+1}$, we have that $2^{l+1}<2^{l+1}+2^{2l}\leq r(x_n)+r(x_m)+r(x_n)r(x_m)<2^{l+1}+2^{2l+2}<2^{l+3}$. Thus $a(r(x_nx_m))$ is $l+1$ or $l+2$, neither of which is congruent to $l$ mod 3, a contradiction. Therefore in this first case (namely when $x_n<2^{a(x_n)+{\frac{1}{2}}}$ for all $n$), 
we must have that $r(x_n)$ converges to 0, which immediately implies that $a(x_n)-b(x_n)$ (the `left gap') goes to infinity.
\par Suppose instead that we are in the second case (namely that $x_n\geq 2^{a(x)+{1\over 2}}$ for all $n$), so that  $r(x_nx_m)=\dfrac{r(x_n)+r(x_m)+r(x_n)r(x_m)-1}{2}$
for all $n\neq m$. Suppose that $a(x_n)-c(x_n)$ is bounded. By passing to a subsequence, we may assume that there 
exists a fixed $l\in\mathbb N$ such that $a(x_n)-c(x_n)=l$ for all $n$. Let $2k-2<d\in\mathbb N$ be such that $\dfrac{(2^{k+1}-1)^d}{2^{(k+1)d}}<\dfrac{1}{2}$, and look at the first $d$ terms. We have that $x_j<2^{a(x_j)+1}-2^{c(x_j)}=2^{a(x_j)+1}-2^{a(x_j)-k}=2^{a(x_j)}\dfrac{2^{k+1}-1}{2^k}$, so that we have $x_1x_2\cdots x_d<2^{a(x_1)+\ldots+a(x_d)}\dfrac{(2^{k+1}-1)d}{2^{kd}}<2^{a(x_1)+\ldots+a(x_d)+k-1}$. On the other hand, by assumption, the product is at least $2^{a(x_1)+\cdots +a(x_d)+\frac{d}{2}}>2^{a(x_1)+\cdots +a(x_d)+k-1}$, a contradiction. Therefore we may assume that $a(x_n)-c(x_n)$ is strictly increasing and goes to infinity, which is equivalent to $r(x_n)$ converging to 1.
\par To summarise, we either have $r(x_n)$ converging to 0, which is equivalent to $a(x_n)-b(x_n)$ going to infinity, or $r(x_n)$ converging to 1, which is equivalent to $a(x_n)-c(x_n)$ going to infinity. We distinguish these two cases.\\
\par\textbf{Case 1.} The sequence $(z_n)_{n\geq1}$ converges to 0. In this case, by passing to a subsequence we may assume that the terms of the sequence $(z_n)_{n\geq1}$ have pairwise left to right disjoint supports in base $P_r$ -- note that this implies that all finite sums of $(x_n)_{n\geq1}$ also have minimal base $P_r$. By passing to a subsequence we may assume that all $y_n$ have the same digit in position $e_r(y_n)+1$ in base $P_r$, and that $z_n<\frac{1}{P_r}$ for all $n$. Suppose that there exist $P_r$ terms such that their integer parts end at the same position in base $P_r$, call it $p$. It is easy to see that the integer part of their sum is the sum of their integer parts, which ends at position $p+1$, a contradiction. Therefore we may assume that the terms of the sequence $(y_n)_{n\geq1}$ have left to right disjoint supports in base $P_r$. By exactly the same argument (looking at a sum of two terms only) we can further deduce that the terms of the sequence $(y_n)_{n\geq1}$ have left to right disjoint supports in binary as well. 
\par Assume first that $r(x_n)$ converges to 0. We fix $x_1$ and look at $x_1+x_n$. For $n$ sufficiently large we have $q(x_1+x_n)=0$, because the left gap of the sum is the left gap of $x_n$, while the end position of $\lfloor x_1+x_n\rfloor$ in base $P_s$ is fixed, namely the end position of $y_1$ in base $P_r$. On the other hand, if the fractional part of $x_1$ has end position $a<0$ in base $P_r$ and $n$ is large enough, then 
$\lfloor x_1x_n\rfloor$ has end position $e_r(y_n)+a$ in base $P_r$, which tends to infinity as $n$ tends to infinity. However, due to the fact that the left gap of $x_n$ goes to infinity, we see that for $n$ large enough the left gap of $x_nx_1$ equals the left gap of $x_1$, which will eventually be less than $e_r(y_n)+a$. So $q(x_1x_n)=1$, a contradiction.
\par Assume now that $r(x_n)$ converges to 1. As before, we fix $x_1$ and look at $x_n+x_1$ for $n$ large enough. Since $x_n$ and $x_1$ have disjoint supports in binary, we have that $a(x_n+x_1)=a(x_n)$, and thus $r(x_n+x_1)=\dfrac{x_n+x_1-2^{a(x_n)}}{2^{a(x_n)}}$ which converges to 1. Therefore, as $n$ tends to infinity, $a(x_n+x_1)-c(x_n+x_1)$ also tends to infinity -- thus it will eventually be greater that the end position of $\lfloor x_n+x_1\rfloor$ in base $P_r$ (which is the end position of $y_1$ in base $P_r$), so $s(x_n+x_1)=0$ for all $n$ large enough. On the other hand, it is a straightforward computation to show that $a(x_nx_1)-c(x_nx_1)$ is either $a(x_1)-c(x_1)$ or $a(x_1)-c(x_1)+1$, and thus is 
bounded. However, we have seen above that $e_r(\lfloor x_nx_1\rfloor)$ is unbounded. We conclude that for all $n$ sufficiently large we have $a(x_nx_1)-c(x_nx_1)<e_r(\lfloor x_nx_1\rfloor)$, and thus $s(x_nx_1)=1$ for all $n$ sufficiently large, a contradiction. This concludes Case 1.\\
\par\textbf{Case 2.} The sequence $(z_n)_{n\geq1}$ converges to 1. In this case we have that $x_n=y_n+1-(1-z_n)$ and the sequence $(1-z_n)_{n\geq1}$ converges to 0. With the same type of argument as the one presented above, we may assume that the terms of the sequence $(y_n+1)_{n\geq1}$ have pairwise left to right disjoint supports in binary and in base $P_r$, and the sequence is strictly increasing (it suffices to show that we cannot have infinitely many terms ending at the same place in binary or in base $P_r$). Since the full argument for base $P_r$ has been given above, here we just include the argument for binary. So suppose that we have $n\neq m$ such that $e_2(y_n+1)=e_2(y_m+1)=p$ and $y_n+1$ and
$y_m+1$ have the same binary digit in position $p+1$ (which we can achieve by passing to a subsequence). Then $e_2(\lfloor x_n\rfloor +1)=p$, while $e_2(\lfloor x_n+x_m\rfloor +1)=e_2(y_n+y_m+2)=p+1$, a contradiction.
\par We observe that for any $n>1$, $e_r(\lfloor x_n+x_1\rfloor +1)=e_r((y_n+1)+(y_1+1))=e_r(y_1+1)$. Let $e_r(x_1)=u<0$ and pick $n$ such that $1-z_n<\frac{1}{x_1}$ and $e_r(y_n+1)=v_n>-u$. This implies that $0<1-(1-z_n)x_1<1$ and that $(y_n+1)x_1\in\mathbb N$. Therefore  $x_nx_1=((y_n+1)-(1-z_n))x_1=(y_n+1)x_1-(1-z_n)x_1$, and thus $\lfloor x_nx_1\rfloor +1=\lfloor x_nx_1+1\rfloor=\lfloor(y_n+1)x_1 + 1-(1-z_n)x_1\rfloor=(y_n+1)x_1$. This means that $e_r(\lfloor x_nx_1\rfloor+1)=v_n+u$ for all $n$ sufficiently large, 
so that the sequence $(e_r(\lfloor x_nx_1\rfloor))_{n\geq1}$ is unbounded.
\par To complete the proof, we show that if $x_n<2^{a(x_n)+\frac{1}{2}}$ for all $n\geq 1$ then for sufficiently large $n$ we have $q'(x_n+x_1)=0$ and $q'(x_nx_1)=1$, while if $x_n\geq 2^{a(x_n)+\frac{1}{2}}$ for all $n\geq1$ then for sufficiently large $n$ 
we have $s'(x_n+x_1)=0$ and $s'(x_nx_1)=1$.
\par Assume first that $x_n<2^{a(x_n)+\frac{1}{2}}$ for all $n\geq1$. As we have seen above, this implies that $a(x_n)-b(x_n)$ tends to infinity (and we may also assume that it is strictly increasing and $a(x_1)-b(x_1)>2$). Consequently $a(x_n+x_1)-b(x_n+x_1)$ also tends to infinity, and so is eventually larger
than $e_r(\lfloor x_n+x_1\rfloor +1)$, whence $q'(x_n+x_1)=0$ for $n$ large enough. On the other hand, since $2^{a(x_n)}+2^{b(x_n)}\leq x_n<2^{a(x_n)}+2^{b(x_n)+1}$ and $2^{a(x_1)}+2^{b(x_1)}\leq x_1<2^{a(x_1)}+2^{b(x_1)+1}$, we have that $2^{a(x_n)+a(x_1)}+2^{a(x_n)+b(x_1)}<x_nx_1<2^{a(x_n)+a(x_1)}+2^{a(x_n)+b(x_1)+1}+2^{a(x_1)+b(x_n)+1}+2^{b(x_n)+b(x_1)+2}
<2^{a(x_n)+a(x_1)}+2^{a(x_n)+b(x_1)+2}$. This is because $b(x_n)+b(x_1)+2<a(x_1)+b(x_n)+1<a(x_n)+b(x_1)+1$.
Therefore $b(x_nx_1)$ is either $a(x_n)+b(x_1)$ or $a(x_n)+b(x_1)+1$, and thus $a(x_nx_1)-b(x_nx_1)\leq a(x_1)-b(x_1)$. Since $e_r(\lfloor x_1x_n\rfloor+1)$ will eventually be greater than $a(x_1)-b(x_1)$, we have that $q'(x_nx_1)=1$ for $n$ large enough, a contradiction.
\par Finally, assume that $x_n\geq 2^{a(x_n)+\frac{1}{2}}$ for all $n\geq1$. Thus $a(x_n)-c(x_n)$ goes to infinity (and as above we may assume it to be strictly increasing and such that $a(x_1)-c(x_1)>2$), 
and consequently so does $a(x_n+x_1)-c(x_n+x_1)$. This means that $a(x_n+x_1)-c(x_n+x_1)>e_r(\lfloor x_n+x_1\rfloor +1)$ for $n$ large enough, and so $s'(x_n+x_1)=0$ for $n$ large enough. On the other hand, $2^{a(x_n)+1}-2^{c(x_n)+1}\leq x_n<2^{a(x_n)+1}-2^{c(x_n)}$ and $2^{a(x_1)+1}-2^{c(x_1)+1}\leq x_1<2^{a(x_1)+1}-2^{c(x_1)}$. This implies that $2^{a(x_n)+a(x_1)+2}-2^{a(x_n)+c(x_1)+3}\leq 2^{a(x_n)+a(x_1)+2}-2^{a(x_n)+c(x_1)+2}
-2^{a(x_1)+c(x_n)+2}+2^{c(x_n)+c(x_1)+2}<x_nx_1<2^{a(x_n)+a(x_1)+2}-2^{a(x_n)+c(x_1)+1}$.
Here the first inequality holds because $a(x_n)+c(x_1)+2>a(x_1)+c(x_n)+2$, which implies that $2^{a(x_n)+c(x_1)+2}+2^{a(x_1)+c(x_n)+2}<2^{a(x_n)+c(x_1)+3}$. Therefore $c(x_nx_1)$ is either $a(x_n)+c(x_1)+1$ or $a(x_n)+c(x_1)+2$, and so  $a(x_nx_1)-c(x_nx_1)\leq a(x_1)-c(x_1)$. Since $e_r(\lfloor x_1x_n\rfloor+1)$ will eventually be greater than $a(x_1)-c(x_1)$, we have that $s'(x_nx_1)=1$ for $n$ large enough, a contradiction. This concludes Case 2.
\end{proof}

Note that Theorem~\ref{finitesums}, together with 
Theorem~\ref{wasunlabelled}, completes the proof of our main
result.
\begin{theorem}\label{mainresult}
There exists a finite colouring of the rational numbers with the property that there exists no sequence such that the set of its finite sums and products is monochromatic and the set of primes that divide the denominators of its terms 
is finite.\hfill$\square$
\end{theorem}

\section{Concluding remarks}
The first remaining problem is of course to understand what happens
with finite sums and products in the rationals. The above colourings of $\mathbb Q_{(k)}$ do rely heavily on the representation of numbers in a suitable base, and so do not pass to sequences from the whole of $\mathbb Q$.

It would be very good to find `parameters' $a$ and $b$ that would allow Lemma~\ref{Marialemma1} to be applied, or perhaps some
variant like Lemma~\ref{Marialemma2}. We have
tried to find such parameters in the rationals in general, but
have been unsuccessful. It would be extremely interesting to
decide whether or not such parameters do exist.

\Addresses

\section*{Appendix}
Here we provide the cases in the proof of Theorem~\ref{0infinity} when the colour class is $C_3$ or $C_5$. 
\begin{proposition} There does not exist an injective sequence $(x_n)_{n\geq1}$ in $\mathbb R^+$ such that the set of all its pairwise sums and products is contained in $C_3=\{2^k+2^l:k,l\in\mathbb Z$ and $l<k\}$.
\end{proposition}
\begin{proof}
Assume for a contradiction that such a sequence $(x_n)_{n\geq1}$ exists. It is easy to see that if $x<y<z$ are three positive real such that $\{x+y,x+z,y+z\}\subseteq C_3$  then
$\{x,y,z\}\subseteq\mathbb Q_{(2)}$, and so $x_n\in\mathbb Q_{(2)}$ for all $n\geq1$.
\par We know that the set $\{x_n:n\in\mathbb N\}\cap\{2^k:k\in\mathbb Z\}$ is finite, otherwise we get a contradiction as the product of two powers of 2 does not lie in $C_3$. We may therefore assume that no $x_n$ is a power of 2.
\par Suppose first that $x_n\in(0,1)$ for all $n\geq 1$. Suppose that $\{s_2(x_n):n\in\mathbb N\}$ is infinite. We may pick $n$ such that $s_2(x_n)<e_2(x_1)$, but then the binary expansion of $x_1+x_n$ has at least four nonzero digits, and thus $x_1+x_n\notin C_3$, a contradiction. We may therefore assume (after passing to a subsequence) that there exists $k\in\mathbb Z$ (with $k<0$) such that $s_2(x_n)=k$ for every $n\geq1$. Then each $x_n=2^k+y_n$ where $s_2(y_n)<k$. Since there are only finitely
many numbers with given values of $s_2(x)$ and $e_2(x)$, by passing to a subsequence we may also assume that $e_2(y_n)>e_2(y_{n+1})$ for all $n\geq1$. We now observe that if $n<m$ then $s_2(x_n+x_m)=k+1$ and $e_2(x_n+x_m)=e_2(x_m)$, so $x_n+x_m$ has a nonzero digit at positions $k+1$ and  $e(x_m)$, and thus, since it is in $C_3$, we have $x_n+x_m=2^{k+1}+2^{e(x_m)}$. But then
$x_1+x_3=x_2+x_3$, a contradiction.
\par We may therefore assume that $x_n>1$ for all $n\geq 1$. By Ramsey's theorem for pairs, we may assume either that for all $n\neq m$ we have $x_n+x_m\in\{2^k+2^l:k,l\in \mathbb Z$ and $0\leq l<k\}$ or that for all $n\neq m$ we have $x_n+x_m\in\{2^k+2^l:k,l\in \mathbb Z$ and $l<0<k\}$.\\

\par\textbf{Case 1.} For all $n\neq m$ we have $x_n+x_m\in\{2^k+2^l:k,l\in \mathbb Z$ and $0\leq l<k\}$.
\par Let $y_n=\lfloor x_n\rfloor$ and $\alpha_n=x_n-y_n$ for all $n\geq1$. Given $n\neq m$, we have $x_n+x_m=y_n+y_m+\alpha_n+\alpha_m$, and so $\alpha_n+\alpha_m\in \{0,1\}$. If $n$, $m$ and $r$ are pairwise distinct and $\alpha_n,\alpha_m,\alpha_r\notin\{0,\frac{1}{2}\}$, then some two are in $(0,\frac{1}{2})$ or some two are in $(\frac{1}{2},1)$, a contradiction. Hence, for all but at most two values of $n$, we have $\alpha_n\in\{0,{1\over 2}\}$. If $n\neq m$ and $\alpha_n=\alpha_m=\frac{1}{2}$, then $x_n\cdot x_m\notin\mathbb N$, again a contradiction. We may therefore assume that $\alpha_n=0$ for all $n\geq1$.
\par Since no $x_n$ is a power of 2, $\{e_2(x_n):n\in\mathbb N\}$ is finite. The reasoning is similar to that presented above: if $e_2(x_n)>s_2(x_1)$ then the binary expansion of $x_1+x_n$ has at least four nonzero digits. We may therefore assume that there exists $k$ such that $e_2(x_n)=k$ for all $n\geq1$. By passing to a subsequence, we may further assume that either each $x_n$ end in $01$ or each $x_n$ ends in $11$, so that $e_2(x_n+x_m)=k+1$. Moreover, we may also assume that $s_2(x_n)<s_2(x_{n+1})$ for all $n\geq1$.
\par We now see that if $n<m$ then $s_2(x_n+x_m)=s_2(x_m)$ or $s_2(x_n+x_m)=s_2(x_m)+1$. Pick $i\neq j$ in $\{1,2,3\}$ and $t\in\{0,1\}$ such that $s_2(x_i+x_4)=s_2(x_4)+t$ and $s_2(x_j+x_4)=s_2(x_4)+t$. Since $k+1<s_2(x_4)+t$ are two positions of nonzero digits, we must have $x_i+x_4=x_j+x_4=2^{s_2(x_4)+t}+2^{k+1}$, a contradiction\\
\par\textbf{Case 2.} For all $n\neq m$ we have $x_n+x_m\in\{2^k+2^l:k,l\in \mathbb Z$ and $l<0<k\}$.
\par In this case, for all $n\neq m$, $x_n+x_m$ has one nonzero digit to the right of the decimal point and one nonzero digit to the left of the decimal point.
\par Suppose first that $\{e_2(x_n):n\in\mathbb N\}$ is unbounded. By passing to a subsequence, we may assume that $0>e_2(x_1)>e_2(x_2)>e_2(x_3)$. This implies that $x_1+x_3$ and $x_2+x_3$ each have a nonzero digit in position $e_2(x_3)$ and $x_1+x_2$ has a nonzero digit in position $e_2(x_2)$. Thus there exist $y,z,w\in\mathbb N$ such that $x_1+x_3=y+2^{e_2(x_3)}$, $x_2+x_3=z+2^{e_2(x_3)}$, and $x_1+x_2=w+2^{e_2(x_2)}$. Clearly we have that $y\neq z$. If $z>y$, then $x_2-x_1=z-y$ so $2x_2=z-y+w+2^{e_2(x_2)}$, whence $e_2(x_2)=e_2(2x_2)=e(x_2)+1$, a contradiction. If $y>z$, then $x_1-x_2=y-z$, so $2x_1=y-z+w+2^{e_2(x_2)}$, giving $e_2(x_2)=e_2(2x_1)=e_2(x_1)+1>e_2(x_2)$, again a contradiction.
\par Hence $\{e_2(x_n):n\in\mathbb N\}$ is bounded. Thus $\{s_2(x_n):n\in\mathbb N\}$ has to be unbounded. We may therefore assume that there exists $k<-1$ such that $e_2(x_n)=k$ for all $n\geq 1$. (If $e_2(x_n)=e_2(x_m)=-1$ then $x_n+x_m\in\mathbb N$.)
By passing to a subsequence, we may also assume that all terms of the sequence have the same digit in position $k+1$, and for all $n\neq m$ we have $e_2(x_n+x_m)=k+1$.
\par We may further assume that $s_2(x_1)<s_2(x_2)<s_2(x_3)<s_2(x_4)$. For $i\in\{1,2,3\}$, $x_i+x_4$ has a nonzero digit in position $s_2(x_4)$ or in position $s_2(x_4)+1$. Pick $i\neq j$ in $\{1,2,3\}$ and $t\in \{0,1\}$ such that $x_i+x_4$ and $x_j+x_4$ each have a nonzero digit in position $s_2(x_4)+t$. Then $x_i+x_4=x_j+x_4=2^{s_2(x_4)+t}+2^{k+1}$, a contradiction.
\end{proof}
\begin{proposition} There does not exist an injective sequence $(x_n)_{n\geq1}$ in $\mathbb R^+$ such that the set of all its pairwise sums and products is contained in $C_5=\{2^{k+1}(1-2^{l-k})^{\frac{1}{2}}:k,l\in\mathbb Z$ and $l<k\}$.
\end{proposition}
\begin{proof} Assume for a contradiction that such a sequence $(x_n)_{n\geq1}$ exists. Let $\alpha$, $\beta$, $\gamma$ be three numbers in $C_5$ such that $x_1+x_2=\alpha$, $x_1+x_3=\beta$ and $x_2+x_3=\gamma$. Let also $x_1x_2=\mu$, $x_1x_3=\nu$ and $x_2x_3=\eta$, where $\mu$, $\nu$ and $\eta$ are in $C_5$. We therefore have $x_1^2=\frac{\mu\cdot\nu}{\eta}$, whence $x_1^4$ is rational.\\

\par\textbf{Case 1.} Suppose that $\alpha\cdot\beta$, $\alpha\cdot\gamma$ and $\beta\cdot\gamma$ are all irrational. Since $\alpha^2$, $\beta^2$ and $\gamma^2$ are rational, $\alpha/\beta$, $\alpha/\gamma$ and $\beta/\gamma$ are all irrational as well. It is easy to show that if $K$ and $R$ are two fields such that $\mathbb Q\subset K\subset R$ and $\delta\in R\setminus K$ is such that $\delta^2\in\mathbb Q$, then $K(\delta)=\{a+b\cdot\delta:a,b\in K\}$. Using this fact, it is straightforward to show that $\beta\notin\mathbb Q(\alpha)$, $\alpha\notin\mathbb Q(\beta)$ and $\gamma\notin\mathbb Q(\alpha,\beta)$.
\par Now, we know that $x_1^4$ is rational. On the other hand, $x_1=\frac{\alpha+\beta-\gamma}{2}$, and so $16\cdot x_1^4=(\alpha+\beta-\gamma)^4=r_0+r_1\cdot\alpha\cdot\beta-r_2\cdot\alpha\cdot\gamma-r_3\cdot\beta\cdot\gamma$, where $r_0$, $r_1$, $r_2$, and $r_3$ are positive rationals. It then follows that $\gamma\cdot(r_2\cdot\alpha + r_3\cdot\beta)=-16\cdot x_1^4 +r_0+r_1\cdot\alpha\cdot\beta$, which implies that $\gamma$ is in $\mathbb Q(\alpha,\beta)$, a contradiction. (For the conscientious reader, the coefficients are $r_0=\alpha^4+\beta^4+\gamma^4+6\cdot\alpha^2\cdot\beta^2+6\cdot\alpha^2\cdot \gamma^2+6\cdot\beta^2\cdot\gamma^2$, $r_1=4\cdot\alpha^2+4\cdot\beta^2+12\cdot\gamma^2$, $r_2=4\cdot\alpha^2+4\cdot \gamma^2+12\cdot\beta^2$ and $r_3=4\cdot\beta^2+4\cdot\gamma^2+12\cdot\alpha^2$.)\\
\par\textbf{Case 2.} Suppose now that $\alpha\cdot\beta$ is a rational number, say $q$. It is clear that $q>0$. We then have $(x_1+x_2)(x_1+x_3)=q=x_1^2+x_1x_3+x_1x_2+x_2x_3=x_1^2+\mu+\nu+\eta$. We now observe that, by the definition of $C_5$, all of its elements are square roots of positive rational numbers. Hence there exist three positive rational numbers $q_1$, $q_2$ and $q_3$, such that $\mu=\sqrt{q_1}$, $\nu=\sqrt{q_2}$ and $\eta=\sqrt{q_3}$. Moreover, since $x_1^2=\frac{\mu\cdot\nu}{\eta}$, it follows that $x_1^2$ is also a square root of a positive rational. More precisely $x_1^2=\sqrt{q_4}$ where $q_4=\frac{q_1\cdot q_2}{q_3}$.
\par We therefore have $q=\sqrt{q_1}+\sqrt{q_2}+\sqrt{q_3}+\sqrt{q_4}$. Let $M=\mathbb Q(\sqrt{q_1}, \sqrt{q_2}, \sqrt{q_3}, \sqrt{q_4})$, and let $d$ be its degree over $\mathbb Q$. On the one hand, the trace of $q$ is $d\cdot q$, and on the other had it is the sum of $d\cdot\sqrt{q_i}$ for those $q_i$ that are perfect squares. This is because, for any positive rational $t$, the trace of $\sqrt{t}$ is 0 if $t$ is not a perfect square, and $d\sqrt{t}$ if $t$ is a perfect square. The only way to have equality in the above is if all the $q_i$ are perfect squares, but then $x_1x_2\in C_5$ is rational, a contradiction.
\end{proof}

\begin{thebibliography}{99}
\bibitem{BHL} V. Bergelson, N. Hindman and I. Leader, Additive and multiplicative Ramsey theory in the reals and the rationals. \textit{Journal of Combinatorial Theory, Series A}, \textbf{85} (1999), 41—68.
\bibitem{MB} M. Bowen and M. Sabok, manuscript (2022).
\bibitem{FST} N. Hindman, Finite sums from sequences within cells of a partition of N. \textit{Journal of Combinatorial Theory, Series A}, \textbf{17} (1974), 1—11.
\bibitem{H} N. Hindman, Partitions and pairwise sums and products. \textit{Journal of Combinatorial Theory, Series A}, \textbf{37} (1984), 46-60.
\bibitem{HS} N. Hindman and D. Strauss, Algebra in the {S}tone-\v{C}ech {C}ompactification: {T}heory and {A}pplications, 2nd edition. \textit{Walter de Gruyter \& Co., Berlin} (2012).
\bibitem{M} J. Moreira, Monochromatic sums and products in $\mathbb N$. \textit{Annals of Mathematics}, \textbf{185} (2017), 1069—1090.

\end{thebibliography}
\end{document}